\documentclass[reqno]{amsart}
\usepackage{color}
\usepackage{dsfont}
\usepackage{amsfonts,latexsym,amssymb,amscd,amsxtra}
\usepackage{ifthen}
\usepackage{graphicx,bm}

\newtheorem{thm}{Theorem}[section]
\newtheorem{lem}{Lemma}[section]
\newtheorem{cor}{Corollary}
\theoremstyle{definition}
\newtheorem{definition}[thm]{Definition}

\newtheorem{prop}{Proposition}[section]
\newtheorem{remark}[thm]{Remark}
\newtheorem{teo}{Theorem}
\newtheorem{problem}{Problem}
\newtheorem{theorem}{Theorem}
\newtheorem{lemma}{Lemma}

\newtheorem{exe}{Exercise}
\newtheorem{exa}{Example}

\newtheorem{question}{Question}

\newtheorem{conjecture}{Conjecture}

\newcommand{\blem}{\begin{lemma}}
\newcommand{\elem}{\end{lemma}}
\newcommand{\bexer}{\begin{exe}}
\newcommand{\eexer}{\end{exe}}
\newcommand{\beq}{\begin{eqnarray}}
\newcommand{\eeq}{\end{eqnarray}}
\newcommand{\bthm}{\begin{theorem}}
\newcommand{\ethm}{\end{theorem}}
\newcommand{\beg}{\begin{exa}}
\newcommand{\eeg}{\end{exa}}

\newcommand{\bdefe}{\begin{definition}}
\newcommand{\edefe}{\end{definition}}

\newcommand{\bprop}{\begin{prop}}
\newcommand{\eprop}{\end{prop}}

\newcommand{\bpf}{\begin{proof}}
\newcommand{\epf}{\end{proof}}

\def\be{\begin{equation}}
\def\ee{\end{equation}}
\newtheorem{cl}{Claim}
\newcommand{\bcl}{\begin{cl}}
\newcommand{\ecl}{\end{cl}}

\newcommand{\B}{\mathbb{B}}
\newcommand{\R}{\mathbb{R}}

\newcommand{\bquess}{\begin{questions}}
\newcommand{\equess}{\end{questions}}
\newcommand{\br}{\begin{remark}}
\newcommand{\er}{\end{remark}}
\newcommand{\brs}{\begin{remarks}}
\newcommand{\ers}{\end{remarks}}

\newcommand{\bn}{\begin{nonsec}}
\newcommand{\en}{\end{nonsec}}

\newcommand{\ds}{\displaystyle}

\newcommand{\bcor}{\begin{cor}}
\newcommand{\ecor}{\end{cor}}
\newcommand{\bprob}{\begin{problem}}
\newcommand{\eprob}{\end{problem}}
\newcommand{\bcon}{\begin{conjecture}}
\newcommand{\econ}{\end{conjecture}}
\newcommand{\bques}{\begin{question}}
\newcommand{\eques}{\end{question}}

\newcommand{\begs}{\begin{examples}}
\newcommand{\eegs}{\end{examples}}

\newcommand{\bdefes}{\begin{definitions}}
\newcommand{\edefes}{\end{definitions}}

\newcommand{\ba}{\begin{array}}
\newcommand{\ea}{\end{array}}

\newcommand{\beqq}{\begin{eqnarray*}}
\newcommand{\eeqq}{\end{eqnarray*}}
\newcommand{\bee}{\begin{enumerate}}
\newcommand{\eee}{\end{enumerate}}

\newcommand{\bei}{\begin{itemize}}
\newcommand{\eei}{\end{itemize}}
\newcommand{\bed}{\begin{description}}
\newcommand{\eed}{\end{description}}

\newcommand{\bo}{\begin{obser}}
\newcommand{\eo}{\end{obser}}
\newcommand{\bos}{\begin{obsers}}
\newcommand{\eos}{\end{obsers}}

 \newcommand{\Sp}{{\mathbb S^{n-1}}}

\numberwithin{equation}{section}

\begin{document}
%Jevti\'c-
%\author{A. Khalfallah, B. Purti\'c  and M. Mateljevi\'c}
\title[Schwarz-Pick Lemma]{Schwarz-Pick Lemma for Harmonic and Hyperbolic Harmonic Functions}
\author{Adel Khalfallah}
\address{A. Khalfallah, Department of Mathematics, King Fahd University of Petroleum and
	Minerals, Dhahran 31261, Saudi Arabia}\email{khelifa@kfupm.edu.sa}

\author{Bojana Purti\'c}
\address{B. Purti\'c, Faculty of mathematics, University of Belgrade, Studentski Trg 16, Belgrade, Republic of Serbia}
\email{bojanaj@matf.bg.ac.rs}

\author{ Miodrag Mateljevi\'c}
\address{M. Mateljevi\'c, Faculty of mathematics, University of Belgrade, Studentski Trg 16, Belgrade, Republic of Serbia}
\email{miodrag@matf.bg.ac.rs}

\begin{abstract} We establish some inequalities of Schwarz–Pick type for harmonic and hyperbolic harmonic functions on the unit ball of $\mathbb{R}^n$ and we disprove a recent conjecture of  Liu \cite{Liu2}.
\end{abstract}

\maketitle
\section{Introduction}
By    $\omega_n$ or $V({\B}^n)$ we denote the $n$-volume of the unit ball  $\mathbb{B}^n$ in $\mathbb{R}^n$, and by $\sigma_n$ the $(n-1)$-volume  of the unit sphere  $\Sp$; note that $\sigma_n=n\omega_n $. Next,  $\sigma$  denotes  the  rotation invariant Borel measure on $\Sp$,   $\sigma^0=\sigma/\sigma_n$    and $|.|$ is the Euclidean norm.
Thus  $\sigma^0$ is   the unique rotation invariant normalized  Borel measure on $\Sp$  such that  $\sigma^0(\Sp)=1$. In this paper, the  expressions  $\omega_{n-1}/ \omega_n$  and $\sigma_{n-1} / \sigma_n$ often appear so it is convenient to denote them by  $\omega_*(n)$  and $\sigma_*(n)$ respectively, that is,
$$\omega_*(n):=\frac{\omega_{n-1}}{ \omega_n}, \quad \mbox{ and }\quad \sigma_*(n):=\frac{\sigma_{n-1}}{ \sigma_n}.$$

Recall that a mapping $u \in \mathcal{C}^2(\B^n,\R)$ is said to be hyperbolic harmonic if 
$\Delta_h u=0,$
where $\Delta_h$ is the hyperbolic Laplacian operator defined by
$$\Delta_h u(x)= (1-|x|^2)^2 \Delta u+ 2(n-2)(1-|x|^2) \sum_{i=1}^n x_i \frac{\partial u}{\partial x_i}(x),  $$
here $\Delta$ denotes the Laplacian on $\R^n$. Clearly for $n=2$, hyperbolic harmonic and harmonic functions coincide.\\

In  \cite{liu}, Liu   proved the Khavinson conjecture,
which says for bounded harmonic functions on the unit ball of $\mathbb{R}^n$ the sharp constants
in the estimates for their radial derivatives and for their gradients coincide. 

\begin{teo}[\cite{liu}]
For $n\geq 3$, if $u$ is a bounded harmonic function on $\mathbb{B}^n$ into $\R$, then we have the following sharp inequality

$$|\nabla u(x)|\leq \frac{c_n}{1-|x|^2} \Phi_n(|x|) |u|_\infty, \, x\in \mathbb{B}^n,$$
with $c_n=(n-1)\omega_*(n),$ and 
$$\Phi_n(r)= \int_{-1}^1 \frac{\left|t-\frac{n-2}{2}r\right|(1-t^2)^{\frac{n-3}{2}} }{(1-2tr+r^2)^{\frac{n-2}{2}}} \, dt.$$
\end{teo}
 
 For more details and development regarding the Khavinson conjecture for harmonic functions, see \cite{Kha92,KM10a,KM10b,KM12,Mar17,Matkhal, Mel19}.\\

In  \cite{Liu2}, the author  further proved that, when $n\geq 4$, the function $\Phi_n$ is decreasing on [0,1], thus
$$ \max_{r\in [0,1]} \Phi_n(r)=\Phi_n(0)=\frac{2}{n-1}.$$
In contrast, if $n= 3$, then

$$\Phi_3(r)= \frac{2}{3} \frac{(1+\frac{1}{3}r^2)^{3/2}-1+r^2}{r^2}$$
is strictly increasing on $[0,1]$ and attains its maximum at $r= 1$, thus
$$ \max_{r\in [0,1]} \Phi_3(r)=\Phi_3(1)=\frac{16}{9\sqrt{3}}.$$

\begin{teo}[\cite{Liu2}, Schwarz-Pick lemma for harmonic functions]\label{ThLiu}
 Let $u$ be a real-valued bounded harmonic function on the unit ball $\mathbb{B}^n$
of $\mathbb{R}^n$.
\begin{enumerate}
    \item When $n = 2$ or   $n\geq  4$, the following sharp inequality holds:
\be \label{InqLiu1}
|\nabla u (x)|\leq 2\omega_*(n) \frac{|u|_{\infty}}{1-|x|^2},  \quad   x\in {\B}^n.
\ee

\noindent  Equality holds if and only
if $x = 0$   and  $u = U \circ T$ for some orthogonal transformation $T$, where $U$ is
the Poisson integral of the function that equals 1 on a hemisphere and $-1$
on the remaining hemisphere.
\item When $n = 3$, we have
\be
|\nabla u (x)|\leq   \frac{8}{3 \sqrt{3}} \frac{|u|_{\infty}}{1-|x|^2}, \quad   x\in {\B}^3 .
\ee
The constant  $\frac{8}{3 \sqrt{3}}$
here is the best possible.
\end{enumerate}
\end{teo}
%\ethm

%Remark 1.
\br \hfill
\begin{enumerate}
    \item Note  that  the inequality (\ref{InqLiu1}) holds  when $n = 3$ at $x=0$.
Curiously, the inequality (\ref{InqLiu1}) fails when $n = 3$ in general. Note that  $\frac{8}{3 \sqrt{3}}\approx 1.5396,$
while the constant $\frac{2 V({\B}^{n-1})}{V({\B}^n)}$ in (\ref{InqLiu1})  equals to  $\frac{3}{2}$  when $n = 3$.
\item The inequality (\ref{InqLiu1}) at $x=0$ was previously proved in \cite[Theorem 6.26]{Axler1992} and in   \cite[Corollary 1]{Burgeth92} for harmonic functions fixing the origin.
\end{enumerate}
\er
%end remark

The classical Schwarz-Pick lemma states that an analytic function of $\mathbb{D}$ into itself satisfies
$$|f'(z)|\leq \frac{1-|f(z)|^2}{1-|z|^2}, z\in \mathbb{D}, $$
where $\mathbb{D}$ denotes the unit disc of the complex plane $\mathbb{C}$.
For complex-valued harmonic function of $\mathbb{D}$ into itself, Colonna \cite{Col89} proved the following sharp Schwarz-Pick lemma:

$$|Df(z)| \leq \frac{4}{\pi} \frac{1}{1-|z|^2}, z\in \mathbb{D},
$$
where $|Df(z)|= \left|\frac{\partial f(z)}{\partial z}\right| +\left|\frac{\partial f(z)}{\partial \overline{z}}\right|$.\\

In the planar case, Kalaj and
Vuorinen \cite{kavu} obtained the following  inequality for real  harmonic functions with values in $(-1,1)$.
\be\label{Schw-Pick}
|\nabla u (z)|\leq \frac{4}{\pi}  \frac{1-|u(z)|^2}{1-|z|^2},\quad |z|<1.
\ee

\noindent Based on (\ref{InqLiu1}) and  (\ref{Schw-Pick}),  Liu  suggested the following conjecture.

\bcon[\cite{Liu2}]\label{ConL}
If $n\geq 4$ and $u:\mathbb{B}^n\rightarrow (-1,1)$  is a harmonic function, then
\be\label{liuconj}
|\nabla u (x)|\leq 2 \omega_*(n)  \frac{1-u^2(x)}{1-|x|^2},\quad |x|<1.
\ee
\econ

First, by providing a counter-example, we  {\it disprove} Conjecture \ref{ConL}  for $n\geq 4$. Our main tool is  Theorem \ref{KMM1} giving a sharp estimate of the norm of the gradient at zero of functions having generalized Poisson transformations; such estimate is based on the Burgeth's method, see \cite{Burgeth92,Burgeth1994}.\\

Let us introduce some notations. 
 If $x\in \B^n\setminus \{0\}$, define $\hat{x}=\frac{x}{|x|} \in \Sp$ and $\hat{0}=e_n=(0,\ldots,1)$, the north pole. $S(\hat{x},\gamma)$
 denotes the hyperspherical cap with center $\hat{x}$ and contact angle $\gamma\in [0,\pi]$:
 $$S(\hat{x},\gamma)=\{y\in \Sp\, :\,  \, \langle \hat{x},y \rangle >\cos \gamma  \}. $$

For $\alpha,\beta \in \mathbb{R}$,  $\beta>0$, the generalized Poisson kernel is defined by 
\begin{equation*}
P_{\alpha,\beta}(x,y)=\frac{(1-|x|^2)^\alpha}{|x-y|^{2\beta}}, \, x\in \B^n \mbox{ and } y\in \Sp.
\end{equation*}

\noindent For $f\in L^1(\Sp)$, set
$$P_{\alpha,\beta}[f](x)= \int_{\Sp} P_{\alpha,\beta}(x,y)f(y)\,d\sigma^0(y). $$

By  $A=A_n^{cap}(\gamma)$  we denote  the normalized $(n-1)$-dimensional volume of the spherical cap  with  contact angle $\gamma$.
% Since the estimate (\ref{estgrad}) in Theorem \ref{KMM1} is sharp

\begin{teo} [\cite{KhMaMh}]\label{KMM1}
Let $h^*: \Sp \to [-1,1]$ be a bounded function on  $\Sp$ with values in $[-1,1]$ and $h=P_{\alpha,\beta}[h^*]$. Then,

\begin{equation}\label{estgrad}
|\nabla h(0)| \leq D_n(\gamma,\beta):=\frac{4\beta}{n}\, \frac{\omega_{n-1}}{\omega_n} (\sin \gamma)^{n-1},
\end{equation}
where $\gamma$ is the unique angle in $ [0,\pi]$ such that
\be\label{constraint}
A=A_n^{cap}(\gamma)=\frac{1+h(0)}{2}.
\ee

The estimate (\ref{estgrad}) is sharp and 
\be\label{countex}
h^0_{\alpha,\beta}:=P_{\alpha,\beta}\left[\mathds{1}_{S(e_n,\gamma)}-\mathds{1}_{S^c(e_n,\gamma)}\right]
\ee
is an  extremal function,
 where $S^c(e_n,\gamma)=\Sp \setminus S(e_n,\gamma)$. 
\end{teo}

\noindent It is readable that  $$D_n(\gamma,\beta)=\sup |\nabla h(0)|,$$ where the  supremum is taken over  all functions $h$ which satisfy the assumptions of Theorem \ref{KMM1} with the constraint $h(0)=a$, where   $\gamma$ is determined by  $A_n^{cap}(\gamma)=\frac{1+a}{2}$.\\

As a corollary, we obtain an estimate of the gradient at zero for harmonic functions obtained for  $(\alpha,\beta)=(1,\frac{n}{2})$ and hyperbolic-harmonic functions obtained for $(\alpha,\beta)= (n-1,n-1)$  in terms of their values at the origin.

\begin{cor}[\cite{KhMaMh}]\label{cor1}
Let $h:\B^n \to (-1,1)$  be a harmonic or hyperbolic harmonic function. Then the following  sharp estimates hold:

$$
|\nabla h(0)| \leq  \left.
\begin{cases}
2 \omega_*(n) (\sin \gamma)^{n-1} & \text{ if } h \text{ is harmonic,} \\
\ds 4 \sigma_*(n) (\sin \gamma)^{n-1} & \text{ if } h \text{ is hyperbolic-harmonic,}
\end{cases}
\right.
$$
where $\gamma$ is the unique angle in $ [0,\pi]$ such that
$$
A_n^{cap}(\gamma)=\frac{1+h(0)}{2}.
$$
These estimates are sharp and $h^0_{1,\frac{n}{2}}$ (resp., $h^0_{n-1,n-1}$)  is an  extremal harmonic (resp., hyperbolic-harmonic) function  on $\B^n$, see (\ref{countex}).
\end{cor}

\section {Main results}

\noindent We are now in a position to establish our main results. Our first result is the following.
\begin{thm}\label{thm12}
For $n\geq 4$, the harmonic function $h^0_{1,\frac{n}{2}}$  defined in (\ref{countex}) provides a counter-example to Conjecture \ref{ConL}.
\end{thm}

The proof is based on Corollary \ref{cor1} and some basic properties of the volume of the unit ball in $\R^n$. Recently, several authors presented interesting monotonicity
properties of the $\omega_n$, the volume of the unit ball in $\R^n$. The sequence itself is not monotonic and attains its maximum at $n=5$.  It is worth noting that    $\omega_*(n)=\frac{\omega_{n-1}}{\omega_n}$ has very interesting properties and 
there are   remarkable upper and lower bounds, see for instance  Borgwardt \cite[p. 253 ]{Borg}   and Alzer  \cite{Alz}. Using these estimates, we prove  the following.

\begin{thm}\label{inq1}
For $n\geq 4$ and $a\in [-1,1]$. Then the following inequalities hold 
\be\label{ineqsin}
(\sin \gamma_a)^{n-1}\geq 1-a^2.
\ee
Moreover, the equality holds for $a=-1,0,1$.
\be\label{inq2}
(\sin \gamma_a)^{n-1}\leq\frac{n-1}{4 \sigma_*(n)} (1-a^2).
\ee
Moreover, the equality holds for $a=-1,1$.\\
 \noindent $\gamma_a$ is the unique angle in $[0,\pi]$ such that 
$A_n^{cap}(\gamma_a)=\frac{1+a}{2}.
$

\end{thm}

\noindent Thus, combining Corollary \ref{cor1} and the inequality (\ref{ineqsin}), we disprove Liu's conjecture.

%For $n=2$,  $D_2(\gamma,1/2)\geq 2 \omega (2) (1-a^2)$,   where  $\omega (2)= 2/\pi$.
\br\hfill
\begin{itemize}
  \item[(i)] For $n=3$, we have  $(\sin \gamma_a)^{n-1}= 1-a^2$.
  \item[(ii)]  For $n=2$, $(\sin \gamma_a)^{n-1}\leq 1-a^2$.
  
\noindent Indeed, in the planar case,  $\ds A_2^{cap}(\gamma_a)=\frac{\gamma_a}{\pi}=\frac{1+a}{2}$,  thus $\ds \gamma_a=\frac{\pi}{2}+\frac{\pi a}{2}$  and  $\sin \gamma_a=\cos(\frac{\pi a}{2})\leq 1-a^2$. 
\end{itemize}
\er

%Using Burgeth method \cite{Burgeth1994} (for $B_3$, but it seems that we can develop   for any $B_n$)   we can derive
%Schwarz lemma for   harmonic functions in several variables.
As for $n=3$, we have $\frac{4}{n} \frac{\omega_{n-1}}{\omega_n}=1$ and combining Theorem \ref{KMM1} and the inequality (\ref{inq2}) in  Theorem \ref{inq1}, it yields the following.

\begin{thm}\label{thm2.4}
Let $n\geq 3$ and  $h^*: \Sp \to [-1,1]$ be a  function on  $\Sp$ with values in $[-1,1]$ and $h=P_{\alpha,\beta}[h^*]$. Then,

\begin{equation}\label{est1}
|\nabla h(0)| \leq \beta (1-|h(0)|^2). 
\end{equation}

The constant $\beta$ is sharp in  (\ref{est1}).
\end{thm}
In particular,  we get the following  estimate of the gradient at zero for harmonic and hyperbolic harmonic functions.

\begin{cor}\label{cor11}
Let $n\geq 3$ and  $h:\B^n \to (-1,1)$  be a harmonic or hyperbolic harmonic function. Then the following  sharp estimates hold:

$$
|\nabla h(0)| \leq  \left.
\begin{cases}
\ds\frac{n}{2} (1-|h(0)|^2) & \text{ if } h \text{ is harmonic,} \\
 (n-1) (1-|h(0)|^2) & \text{ if } h \text{ is hyperbolic-harmonic}.
\end{cases}
\right.
$$

Furthermore, this inequality is strict for $n \geq 4$.
\end{cor}

Using the ball of center $x$ and radius $1-|x|$ in the harmonic case and M\"obius transformations in the hyperbolic-harmonic case, we obtain the following.

\begin{thm}\label{thm2.5}
Let $n\geq 3$ and  $h:\B^n \to (-1,1)$  be a harmonic function, then
$$
|\nabla h (x)|\leq \frac{n}{2}  \frac{1-|h(x)|^2}{1-|x|}.
$$
In addition, this inequality is strict for $n \geq 4$.
\end{thm}

\begin{thm}\label{thm2.6}
Let $n\geq 3$ and let $h:\B^n \to (-1,1)$  be a hyperbolic harmonic function, then
$$
|\nabla h (x)|\leq  (n-1)  \frac{1-
|h(x)|^2
}{1-|x|^2}.
$$
Therefore,
$$
d_{h_2}(h(x_1),h(x_2))\leq  (n-1) \, d_{h_n}(x_1,x_2), \quad   x_1,x_2\in  \mathbb{B}^n,  
$$
\end{thm}
 where $d_{h_n}$ denotes the hyperbolic distance of the unit ball $\B^n$. For $n=2$, $d_{h_2}$ is simply the Poincar\'e distance of the unit disc. \\
 
 This result is connected to the Khavinson conjecture for hyperbolic harmonic functions, see \cite{KhaHagg}.
 
 \begin{teo}\cite[Theorem 2]{KhaHagg}
 Let $n\geq 3$ and let $h:\B^n \to (-1,1)$  be a hyperbolic harmonic function, then
$$
|\nabla h (x)|\leq  4\sigma_*(n) \frac{1}{1-|x|^2}.
$$
 \end{teo}

\noindent  For vector valued functions, we prove the following.
 
 \begin{thm}\label{thm:Dhh}
Let $n\geq 3$, $m\geq 1$ and  $h:\B^n \to \B^m$  be a harmonic or hyperbolic harmonic function. Then the following   estimates hold.

$$
|Dh(x)| \leq  \left.
\begin{cases}
\ds\frac{n}{2}  \frac{1}{1-|x|} & \text{ if } h \text{ is harmonic,} \\
\ds (n-1)  \frac{1}{1-|x|^2} & \text{ if } h \text{ is hyperbolic-harmonic},
\end{cases}
\right.
$$
where
$$|Dh(x)|=\sup_{|v|=1} |Dh(x)v|. $$
\end{thm}
 
\noindent In addition, it yields the following. 
 
 \begin{thm}\label{thm:hh}
Let $n\geq 3$, $m\geq 1$ and  $h:\B^n \to \B^m$  be a harmonic or hyperbolic harmonic function. Then the following   estimates hold:

$$
|\nabla\, |h|(x)| \leq  \left.
\begin{cases}
\ds\frac{n}{2}  \frac{1-|h(x)|^2}{1-|x|} & \text{ if } h \text{ is harmonic,} \\
\ds (n-1)  \frac{1-
|h(x)|^2
}{1-|x|^2} & \text{ if } h \text{ is hyperbolic-harmonic}.
\end{cases}
\right.
$$

Furthermore, this inequality is strict for $n \geq 4$.
\end{thm}

Recall that 
$$
\left|\nabla\, |h|(x)\right|= \sup_{\beta \in \Sp} \lim_{t \to 0^+} \frac{\left|\,  |h(x+t\beta)|-  
|h(x)|\, \right|}{t}.
$$

Thus $\left|\nabla\, |h|(x)\right|$ coincides with the gradient of $|h|$ at $x$, if $h(x)\not =0$. Moreover, if $h(x)=0$, then $\left|\nabla\, |h|(x)\right|= |Dh(x)|$. Therefore, Theorem \ref{thm:hh} can be seen as an extension of Theorem \ref{thm:Dhh}. \\

We mention that in \cite{kavu}, the authors considered the corresponding theorem for vector harmonic functions defined on the unit disc, see \cite[Theorem 1.10]{kavu}. A Schwarz lemma for the modulus of a vector-valued analytic functions was previously  considered in \cite{pav}.\\

In \cite{Burgeth1994}, Burgeth introduced $R_{\alpha,\beta}(|x|,h(x))$ and  provides estimates for the radial derivative   in terms of  $|x|$  and $h(x)$.
In the  three-dimensional case, we  get explicit formula for  $R_{\alpha,\beta}(|x|,h(x))$.
In particular, we can modify   \cite[Corollary 3]{Burgeth1994} in the following way.

\begin{thm} \label{Thyp3}    If $u:\mathbb{B}^3\rightarrow (-1,1)$  is a harmonic function, then   for each $x\in \mathbb{B}^3$,
\be
\frac{-3-|x|u(x)}{1-|x|^2}   (1-u^2(x))/2\leq D_r u(x)\leq \frac{3-|x|u(x)}{1-|x|^2}  (1-u^2(x))/2.
\ee
\end{thm}

Hence
\be
|\nabla u (x)|\leq 2  \frac{1-u^2(x)}{1-|x|^2},
\ee

and therefore

\be
d_{h_2}(u(x_1),u(x_2))\leq 2  d_{h_3}(x_1,x_2), \quad   x_1,x_2\in  \mathbb{B}^3. 
\ee

We should mention that, recently, several versions of Schwarz lemma for harmonic and pluriharmonic
mappings were established, see \cite{kalaj,Che13,Mat18,CH2020,DCP2010,Zhu2019}.\\
 
We close our paper, by the following question as an adjustment to Liu's conjecture.

\bques
 Let $n\geq 3$ and  $h:\B^n \to (-1,1)$  be  a harmonic function. Is it true that
$$
|\nabla h (x)|\leq \frac{n}{2}  \frac{1-|h(x)|^2}{1-|x|^2}, \quad x\in\B^n ?
$$
\eques

\section{Proofs of the main results}
\subsection*{Proof of Theorem \ref{thm12}} 
To simplify notations, let us denote 
$$A=A_n^{cap}(\gamma).$$
Using spherical coordinates, the formula for the normalized  area of the spherical cap with contact angle $\gamma\in[0,\pi]$  is given by
\be\label{Azero}
A= \sigma_*(n)A_0(\gamma),
\ee
where  
$$A_0(\gamma)=\int_0^\gamma (\sin \theta)^{n-2} d\theta,$$  see for example \cite{ahl3}.

\begin{lem}Using the notations of Theorem \ref{KMM1}, for $\gamma\in (0,\pi)$, we have
\be
D_n(\gamma,\beta)= \beta C_n(\gamma) (1-|h(0)|^2),
\ee
where 
\be\label{Cn}
C_n(\gamma)=\dfrac{1}{(n-1)}\dfrac{(\sin \gamma)^{n-1}}{A_0(\gamma) (1-A_n^{cap}(\gamma))},
\ee

\noindent and
$$ C_n(\gamma) \to 1 \mbox { as } \gamma \to 0^+.$$
\end{lem}

\begin{proof} 
By the constraint condition (\ref{constraint}), we have  $1+h(0)=2A$ and  $1-h(0)=2(1-A)$. Therefore  
\be \label{hzero}
1-|h(0)|^2=4 A (1-A).
\ee

Recall that $D_n(\beta,\gamma)=\dfrac{4\beta \omega_*(n)}{n} (\sin\gamma)^{n-1}=\dfrac{4\beta \sigma_*(n)}{n-1} (\sin\gamma)^{n-1}$.

By (\ref{hzero}), we obtain 
$D_n(\beta,\gamma)=\dfrac{\beta}{n-1}\dfrac{\sigma_*(n) (\sin\gamma)^{n-1} }{A(1-A)} (1-|h(0)|^2).$
Therefore, by (\ref{Azero}), we get  (\ref{Cn}). The limit of $C_n(\gamma)$ as $\gamma$ goes to  $0$ equals to $1$ is a direct consequence of  L'Ho\^pital's rule.\\
\end{proof}

First, we disprove Conjecture \ref{ConL}. Assume that  the inequality (\ref{liuconj}) is true. Then,  in particular  for $x=0$, we have 
\be
|\nabla u (0)|\leq  2 \omega_*(n) (1-a^2), 
\ee
where  $a=u(0)$.
Since the estimate (\ref{estgrad}) in Theorem \ref{KMM1} under the constrain  $a=u(0)$ is sharp, then there is  extremal function $u^0$,  such that $D_n(\gamma,\beta) =|\nabla u^0 (0)|$.

Thus
\be
D_n(\gamma,\beta)=\beta C_n(\gamma) (1-a^2) \leq2 \omega_*(n) (1-a^2).
\ee
Therefore $\beta C_n(\gamma) \leq 2  \omega_*(n)$. In particular, in harmonic case where  $\beta=n/2$, we have
$$
n  C_n(\gamma) \leq 4  \omega_*(n).
$$
 As the limit of $C_n(\gamma)$ is $1$ as  $\gamma \to 0$, it yields
 \be\label{eq123}
 n   \leq 4  \omega_*(n).
 \ee
 Since   $\omega_n$ assumes its maximal value when $n=5$, $\omega_*(5)<1$, we disprove  Liu's conjecture.\\

We leave the reader  to disprove  Conjecture \ref{ConL} for $n\geq 4$ using   Alzer's estimate below.

\subsection{Proof of Theorem \ref{inq1}}\hfill
A remarkable upper and lower bounds for the ratio $\frac{\omega_{n-1}}{\omega_n}$ are proved by Borgwardt[7, p.253].  He showed that for $n\geq 2$
\begin{equation}\label{ineqbor}
	\sqrt{\frac{n}{2\pi}} \leq \omega_*(n) \leq \sqrt{\frac{n+1}{2\pi}}.
\end{equation}

\noindent More refinements of these estimates are established by  Alzer \cite{Alz}. 

\begin{thm}\cite{Alz} For $n\geq 2$, we have 
\be\label{eq125}
\sqrt{\frac{n+A}{2\pi}} \leq \omega_*(n) \leq \sqrt{\frac{n+B}{2\pi}},
\ee
with the best possible constants
$$A=\frac{1}{2} \mbox{ and } B=\frac{\pi}{2}-1. $$
\end{thm}

For more properties of the volume of the unit ball in $\R^n$, see \cite{Alz2,Mortici}.  \\

As $\sigma_*(n)=\frac{n-1}{n}\, \omega_*(n)$ and using (\ref{ineqbor}) or (\ref{eq125}), one can easily check the following lemma.

\begin{lem} For $n \geq 4$, we have
	\be\label{eq111}
	 \frac{1}{2}< \sqrt{\frac{n-1}{8}}< \sigma_*(n) < \frac{n-1}{4}.
	 \ee
\end{lem}

%\begin{equation}
%	\sigma_*(n)^2 \geq \frac{(n-1)^2}{2\pi n}, \quad \mbox{ for all } n\geq 2.
%\end{equation}

\noindent Recall that the area of the spherical cap of contact angle
$\gamma \in [0,\pi]$ is given by

\begin{equation}
	A(\gamma)= \sigma_*(n) \int_0^\gamma \sin^{n-2} \theta d\theta.
\end{equation}

Let $a\in [-1,1]$, then there exists a unique angle $\gamma(a)\in[0,\pi]$ such that
\begin{equation}\label{eq124}
	A(\gamma(a))=\frac{1+a}{2}.
\end{equation}

Clearly, the mapping $a \mapsto \gamma(a)$ is strictly  increasing from $[-1,1]$ to $[0,\pi]$.\\

Differentiating  the equation (\ref{eq124}) with respect to $a$, we obtain
\begin{equation}\label{2.12}
	\sigma_*(n) \gamma'(a) \sin^{n-2} \gamma(a) =\frac{1}{2}, \quad \mbox{ for } a\in (-1,1).
\end{equation}

Let us consider the function
$h$ defined on $[-1,1]$ by
\begin{equation}
	h(a):=\sin^{n-1} \gamma(a)-1+a^2.
\end{equation}

As  $h$ is even, it is enough to study the function $h$ on $[0,1].$\\

For each $a\in [0,1)$, we have 
$$
	h'(a)=(n-1) \gamma'(a) \sin^{n-2}\gamma(a)\cos \gamma(a)+2a.
$$
In view of the equation (\ref{2.12}), we get

$$
	h'(a)=\frac{(n-1)\cos \gamma(a)}{2\sigma_*(n)}+2a.
$$
The second derivative of $h$ is given by
$$
	h''(a)=-\frac{n-1}{2\sigma_*(n)}\sin \gamma(a) \gamma'(a)+2.
$$
Again using (\ref{2.12}), we deduce that

$$
	h''(a)=-\frac{n-1}{4\sigma_*(n)^2}\sin^{3-n} \gamma(a) +2.
$$

For $n\geq 4$, we conclude that $h''$ is strictly decreasing on $[0,1]$
because $\gamma$ is increasing with values in $[\pi/2,\pi]$. 
%Thus for $a\in[0,1)$, we have  $h''(a) \leq h''(0)$. 
Moreover, we have
\begin{equation}
	h''(0)=-\frac{n-1}{4\sigma_*(n)^2}+2,  \mbox{ and }\quad \lim_{a\to 1}h''(a)=-\infty.
\end{equation}

Using the inequality (\ref{eq111}), it yields
\begin{equation}
	h''(0) >0, \mbox{ for } n\geq 4.
\end{equation}
Therefore, there exists $a_n\in (0,1)$ such that $h''(a)>0$ on $(0,a_n)$ and $h''(a)<0$ on $(a_n,1)$.
Thus $h'$ is increasing on  $[0,a_n)$ and decreasing on $(a_n,1)$.
Moreover, 
$$h'(0)=0 \mbox{ and } h'(1)=-\frac{n-1}{2\sigma_*(n)}+2.$$

Now, using (\ref{eq111}), we deduce that
$$h'(1)<0.$$

Therefore there exists $b_n\in(0,1)$ such that $h$ is increasing on $(0,b_n)$ and decreasing on $(b_n,1)$.
As $h(0)=h(1)=0$, we conclude  that the $\min_{a\in [0,1]} h(a)=h(0)=0$. Finally, we get 
$
h(a) \geq 0 \mbox{ for all } a\in [0,1]
$, that is,   
$\sin^{n-1} \gamma(a) \geq 1-a^2,$ for $n\geq 4$. \\

\noindent Next, in order to prove the second estimate in Theorem \ref{inq1}, we  need the following lemma.

\begin{lem}\label{lem3.3}
	Let $a\in [0,1]$ and $n\geq 4$. Then
	$$-\cos \gamma (a) \leq a,$$
	Moreover, the equality holds at $a=0$ or $a=1$.
\end{lem}

\begin{proof}Consider the function $G$ on $[0,1]$ defined by 
$$G(a)=a+\cos \gamma(a). $$
 
$G$ is differentiable on $(0,1)$ and
$$ G'(a)=1-\sin \gamma(a) \gamma'(a)=1- \frac{1}{2\sigma_*(n) }\sin^{3-n} \gamma(a).$$

%Thus the sign of $G'(a)$ depends on $$ %H(a):=2\sigma_*(n) \sin^{n-3}\gamma(a)-1.$$
Thus, the function $G'$ is strictly decreasing as $\gamma(a)$ belongs to $[\pi/2,\pi]$. Moreover, we have
$$G'(0)= 1- \frac{1}{2\sigma_*(n)}>0 \mbox{ and} \lim_{a\to 1}G'(a)=-\infty.$$

Therefore, $G$ is strictly increasing on $[0,c_n]$ and decreasing on $(c_n,1)$, where
  $c_n\in (0,1)$ is the unique zero of $G'$. As $G(0)=G(1)=0$,  we get   $G(a)>0$ on $(0,1)$.\\
\end{proof}

It remains  to prove  the following: $(\sin \gamma)^{n-1}\leq\frac{n-1}{4 \sigma_*(n)} (1-a^2)$,
$n\geq 4$.

Consider  $$g(a)=   1-a^2 - \frac{4 \sigma_*(n)}{n-1} \sin^{n-1}
\gamma(a).$$

Then  $$g'(a)= -2a -2\cos \gamma(a).$$

By  Lemma \ref{lem3.3}, we have  $g'(a)\leq 0$  and $g$ is decreasing. Hence
$g(a)\geq g(1)=0$
on $[0,1]$  and the conclusion follows.

\subsection*{Proof of Theorem \ref{thm2.4}} 
Theorem \ref{thm2.4} is a direct consequence of Theorem \ref{KMM1} and Theorem \ref{inq1}. In particular, in Corollary \ref{cor11}, we obtain the corresponding inequality for harmonic 
and hyperbolic harmonic functions, by considering the corresponding  values of $\beta$.

\subsection*{Proof of Theorem \ref{thm2.5} and \ref{thm2.6}} 
Let $h:\B^n \to (-1,1)$ be a harmonic function and let $x\in \B^n$. Consider $g$ the harmonic function defined on $\B^n$ by
$$g(y)=h(x+y(1-|x|)).$$ 
Clearly, we have $g(0)=h(x)$ and $|\nabla g(0)|=(1-|x|)|\nabla h(x)|$. By applying Corollary \ref{cor11} to the function $g$, we get the desired  inequality.\\

In the hyperbolic harmonic case, we compose with a M\"obius transformation sending $0$ to $x$. More precisely, let $x\in \B^n$ be fixed. By the M\"obius invariance of $\Delta_h$, the function $h\circ \varphi_x$ is also a bounded hyperbolic harmonic function, where
$$\varphi_x(y):= \frac{|y-x|^2 x-(1-|x|^2)(y-x)}{1-2\langle y,x\rangle+|y|^2|x|^2}, $$
which is a M\"obius transformation of $\B^n$. Theorem \ref{thm2.6} follows by replacing $h$ by $h\circ \varphi_x$ in Corollary \ref{cor11} and noting that $\nabla(h\circ \varphi_x)(0)=-(1-|x|^2)\nabla h(x)$, see \cite[p. 18]{stoll}.

\subsection{Proof of Theorem \ref{thm:Dhh}}
As the proofs for the harmonic and the hyperbolic harmonic case are similar, we will provide only the proof in the harmonic setting. Let $h:\B^n \to \B^m$ be a harmonic vector-function and $\theta$ be a unit vector in $\R^m$. Consider $h_\theta$ the function defined by  
$$h_\theta (x)=\langle h(x), \theta \rangle. $$
Then $h_\theta$ is a harmonic function with values in $(-1,1)$. Consequently, by Theorem \ref{thm2.5}, we get 
$$|\langle Dh(x)v,\theta \rangle|=|Dh_\theta(x)v| \leq \frac{n}{2}  \frac{1}{1-|x|}, \quad \mbox {for all } v\in \R^n,\, |v|=1.  $$
Therefore,  $|Dh(x)| \leq \frac{n}{2}  \frac{1}{1-|x|}.$

\subsection{Proof of Theorem \ref{thm:hh}}

Let $x_0\in \B^n$ and $h:\B^n \to \B^m$ be a harmonic function. By Theorem \ref{thm:Dhh}, we have to consider only the case where $h(x_0)\not =0$. Define 
$$g(x):=\langle h(x), \frac{h(x_0)}{|h(x_0)|} \rangle.$$
Then $g$ is a harmonic function on $\B^n$ with codomain $(-1,1)$ with $g(x_0)=|h(x_0)|.$ It follows from Theorem \ref{thm2.5}, that 
$$ 
|\nabla g (x_0)|\leq \frac{n}{2}  \frac{1-|g(x_0)|^2}{1-|x_0|}.
$$

\noindent Indeed,  easy computations show that $|\nabla g (x_0)|=|\nabla |h| (x_0)|$ as $ g_{x_i}(x_0)=|h|_{x_i}(x_0)$, where $g_{x_i}(x_0)$ denotes the partial derivative of $g$ with respect to the variable $x_i$ at $x_0$.

%\section{Declarations}
%\noindent {\bf Conflicts of interest}: The authors have no conflicts of interest to declare that
%are relevant to the content of this article.

\end{document}